\pdfoutput=1
\documentclass[12pt]{amsart}
\usepackage{amsmath}
\usepackage{amsthm}
\usepackage{amsfonts}
\usepackage{amssymb}
\usepackage{url}
\makeatletter
\g@addto@macro{\UrlBreaks}{\UrlOrds}
\makeatother
\renewcommand{\UrlBreaks}{\do\/\do\a\do\b\do\c\do\d\do\e\do\f\do\g\do\h\do\i\do\j\do\k\do\l\do\m\do\n\do\o\do\p\do\q\do\r\do\s\do\t\do\u\do\v\do\w\do\x\do\y\do\z\do\A\do\B\do\C\do\D\do\E\do\F\do\G\do\H\do\I\do\J\do\K\do\L\do\M\do\N\do\O\do\P\do\Q\do\R\do\S\do\T\do\U\do\V\do\W\do\X\do\Y\do\Z}

\theoremstyle{plain}
\newtheorem{theorem}{Theorem}
\newtheorem{lemma}{Lemma}

\newcommand{\nequiv}{\not\equiv}

\let\abs=\envert
\let\vph=\varphi

\newcommand{\floor}[1]{\left\lfloor#1\right\rfloor}
\newcommand{\Mod}[1]{\ \left(\mathrm{mod}\ #1\right)}

\newcommand{\Q}{\mathbb{Q}}

\begin{document}
\title[Explicit improvements of the Brun-Titchmarsh theorem]
{Explicit improvements of the Brun-Titchmarsh theorem for arbitrary intervals}
\author{Tomohiro Yamada}
\keywords{Brun-Titchmarsh theorem.}
\subjclass{Primary 11N36, Secondary 11B83, 11N05, 11N13.}
\address{Center for Japanese language and culture, Osaka University,
562-8678, 3-5-10, Sembahigashi, Minoo, Osaka, Japan}
\email{tyamada1093@gmail.com}

\date{}

\begin{abstract}
We give an explicit version of Brun-Titchmarsh theorem applicable for arbitrary moduli and arbitrary intervals.
For example, we show that $\pi(x+y; k, a)-\pi(x; k, a)<2y/(\vph(k)(\log (y/k)+0.8601))$ for any relatively
prime positive integers $k$ and $a$ and any positive real numbers $x$ and $y$ such that $y>k$.
\end{abstract}

\maketitle

\section{Introduction}\label{intro}

Brun-Titchmarsh type theorems state that an inequality
\begin{equation}\label{eq10}
\pi(x+y; k, a)-\pi(x; k, a)<\frac{Cy}{\vph(k)\log(y/k)}
\end{equation}
holds with a constant $C$ uniformly for $k$ and $x$ provided only that $k<y$ and $\gcd(k, a)=1$.

In his epoch-making work \cite{Bru}, Brun used his sieve method to prove that the number of integers $n\leq y$ free of prime factors
below $y^{1/6}$ is between $y/\log y$ and $7y/\log y$ and therefore $\pi(y)<8y/\log y$ for sufficiently large $y$.
Brun's argument is not prevented from applying to the set of integers $n$ such that $x<n\leq x+y$ and we see that
there exists an absolute constant $x_0$ such that
\begin{equation}
\pi(x+y)-\pi(x)<\frac{8y}{\log y}
\end{equation}
for any real $x, y$ with $x\geq x_0$.
Montgomery and Vaughan \cite{MV} showed that \eqref{eq10} holds with $C=2$.
Maynard \cite{May} showed that
there exists an effectively computable constant $k_0$ such that
\begin{equation}
\pi(x; k, a)<\frac{2\operatorname{Li} x}{\vph(k)}
\end{equation}
for $k_0\leq k\leq x^{1/8}$.

For an arbitrary interval $(x, x+y]$, we write $\pi(x, x+y; k, a)=\pi(x+y; k, a)-\pi(x; k, a)$.
Montgomery and Vaughan \cite{MV} showed that there exists
an absolute constant $c$ for which we have
\begin{equation}\label{eq10b}
\pi(x, x+y; k, a)<\frac{2y}{\vph(k)(\log(y/k)+5/6)}
\end{equation}
provided that $y>ck$ without making the constant $c$ explicit.

Upper bounds for $\sup_x \pi(x+y)-\pi(x)$ have been studied by several authors such as
Hensley and Richards \cite{HR}, who showed that we can have
$\pi(x+20000)-\pi(x)>\pi(20000)$ for some $x$ if the prime $k$-tuple conjecture is true,
Clark and Jarvis \cite{CJ}, and Gordon and Rodemich \cite{GR}.
According to Engelsma's page \cite{Eng}, we have $\max_x(\pi(x+y)-\pi(x))=\pi(y)$ for $y\leq 2529$
but $\max_x(\pi(x+3159)-\pi(x))=447>\pi(3159)$ if the prime $k$-tuple conjecture is true.

The purpose of this paper is to give an explicit improvement of results of Montgomery and Vaughan
applicable for arbitrary moduli and arbitrary intervals.

We use inequalities of the type
\begin{equation}\label{eq11}
\sum_{q\leq z}g(q)=C_g(\log z+\eta_g)+E_g(z)
\end{equation}
with constants $C_g$, $\eta_g$, and $c_g$ and a function $E_g(z)$ satisfying $\abs{E_g(z)}<c_g/z^{1/2}$.

Let $k$ be a positive squarefree integer.
Alterman \cite{Alt} proved that, for any positive integer $k$ and any real $z\geq 1$,
\begin{equation}\label{eq12}
\sum_{q\leq z, \gcd(q, k)=1}\frac{\mu^2(q)}{\vph(q)}=\frac{\vph(k)}{k}(\log z+\eta_k)+E_k(z)
\end{equation}
with
\begin{equation}
\eta_k=\gamma+C+\sum_{p\mid k}\frac{\log p}{p}
\end{equation}
and $\abs{E_k(z)}<c_k/z^{1/2}$, where we take
\begin{equation}\label{eq12b}
c_k=4.4\delta_k\prod_{p\mid k}\left(1+\frac{p-2}{p^{3/2}-p-\sqrt{p}+2}\right)
\end{equation}
with $\delta_k=1$ if $k$ is odd and $0.493$ if $k$ is even.
For $k=1$, Ramar\'{e} \cite{Ram} proved that we can take $c_1=2.44$.

Alterman's result allows us to give the following explicit result for general moduli $k$.
\begin{theorem}\label{th11}
Let $k>0$ and $a$ be relatively prime integers.
For a real $z>0$, a positive integer $N$, and a real $y>k$, we define
\begin{equation}
\epsilon_{k, 1}(z)=\frac{\eta_k}{1+z}
+\left(\frac{1}{2}+\frac{\pi}{4}\right)\frac{k c_k}{\vph(k)z^{1/2}},
\end{equation}
\begin{equation}
\epsilon_{k, 2}(N)=\frac{\vph(k)(\log N+2\eta_k-\log (16/3))^2}{k\sqrt{16N/3}},
\end{equation}
\begin{equation}
\epsilon_{k, 3}(y)=\frac{k(\log (y/k)+2\eta_k-\log (16/3))^2}{y}
\end{equation}
respectively and
$\epsilon_k(y)=\epsilon_{k, 1}(\sqrt{3\floor{y/k}/4})+\epsilon_{k, 2}(\floor{y/k})+\epsilon_{k, 3}(y)$.

Then, for positive real numbers $x$ and $y$ with $y>k$, we have
\begin{equation}
\pi(x, x+y; k, a)\leq \frac{2y}{\vph(k)\left(\log(3ye^{2\eta_k}/(16k))-\epsilon_k(y)\right)}.
\end{equation}
\end{theorem}

More explicitly, we have the following explicit estimate.
\begin{theorem}\label{th12}
Let $k\geq 1$ and $a$ be relatively prime integers.
Then, for any positive real numbers $x$ and $y$ with $y>k$, we have
\begin{equation}\label{eq13}
\pi(x, x+y; k, a)<\frac{2y}{\varphi(k)(\log (y/k)+0.8601)}.
\end{equation}
\end{theorem}

For each moduli $k\leq 16$, we have the following explicit estimates.
\begin{theorem}\label{th13}
Let $k\leq 16$ and $a$ be integers such that $k\nequiv 2\Mod{4}$ and $\gcd(k, a)=1$.
Then, for any positive real numbers $x$ and $y$ with $y>k$, we have
\begin{equation}\label{eq14}
\pi(x, x+y; k, a)<\frac{2y}{\varphi(k)(\log (y/k)+\xi_k)}
\end{equation}
with a constant $\xi_k$ given in Table \ref{tbl1}.
\end{theorem}

\begin{table}
\caption{$\xi_k$ ($k\leq 16$)}
\begin{center}
\begin{small}
\begin{tabular}{| c | c |}
 \hline
$k$ & $\xi_k$ \\
\hline
$1$ & $0.8601$ \\
$3$ & $1.5864$ \\
$4$ & $1.5878$ \\
$5$ & $1.4782$ \\
$7$ & $1.3925$ \\
$8$ & $1.5878$ \\
$9$ & $1.5613$ \\
$11$ & $1.2765$ \\
$12$ & $2.3143$ \\
$13$ & $1.2385$ \\
$15$ & $2.2142$ \\
$16$ & $1.5857$ \\
 \hline
\end{tabular}
\label{tbl1}
\end{small}
\end{center}
\end{table}

Our argument proceeds along the line of Montgomery-Vaughan paper \cite{MV}.
We use Alterman's result for sufficiently large $y$, numerical estimates of $H_g(z)$ for moderately large $y$,
and Eratosthenes-Legendre sieve estimates for various finite sets of primes for small $y$.
We used PARI-GP for our calculation.
Our PARI-GP script files are available in
\url{https://drive.google.com/drive/folders/1KW72w7CUN6mkA8UMLnCB69iXsZrLFqhX}.

\section{Preliminaries}

\begin{lemma}\label{lm21}
Assume that \eqref{eq11} holds for any $z\geq 1$.
Then
\begin{equation}\label{eq21}
H_g(z)
=C_g\left(\log z-\log 2+\log\left(1+\frac{1}{z}\right)+\frac{z}{z+1}\eta_g\right)+\Delta_g(z),
\end{equation}
where
\begin{equation}
\Delta_g(z)=z\int_1^z \frac{E_g(t)}{(z+t)^2}dt.
\end{equation}
In particular, if we have $\abs{E_g(t)}<ct^{-1/2}$ for $1\leq t\leq z$, then we have \eqref{eq21} with
\begin{equation}\label{eq22}
\abs{\Delta_g(z)}<\left(\frac{1}{2}+\frac{\pi}{4}\right)\frac{c}{z^{1/2}}.
\end{equation}
\end{lemma}

\begin{proof}
Partial summation gives that
\begin{equation}
\sum_{w<q\leq z}\left(1+\frac{z}{q}\right)^{-1}g(q)
=\frac{G(z)}{2}-\frac{zG(w)}{z+w}+z\int_w^z \frac{G(t)}{(z+t)^2}dt
\end{equation}
for any real $w>0$.
It is clear that $G(w)=0$ for $0<w<1$.
Hence, taking $w\rightarrow 1-0$, we obtain
\begin{equation}\label{eq23}
\sum_{q\leq z}\left(1+\frac{z}{q}\right)^{-1}g(q)
=\frac{G(z)}{2}+z\int_1^z \frac{G(t)}{(z+t)^2}dt.
\end{equation}

We see that
\begin{equation}
\begin{split}
& \int_1^z \frac{G(t)}{(z+t)^2}dt=\int_1^z\frac{c_g(\log t+\eta_g)+E_g(t)}{(z+t)^2}dt \\
& =C_g\left(-\frac{\log z}{2z}+\int_1^z \frac{dt}{t(z+t)}
+\frac{\eta_g}{z+1}-\frac{\eta_g}{2z}\right)+z\int_1^z \frac{E_g(t)}{(z+t)^2}dt
\end{split}
\end{equation}
and the right-hand side of \eqref{eq23} is equal to
\begin{equation}
\begin{split}
& C_g\left(\frac{\log z+\eta_g}{2}+\frac{z\eta_g}{z+1}-\frac{\eta_g}{2}\right)+
z\int_1^z\frac{E_g(t)}{(z+t)^2}dt \\
& +C_g\left(-\frac{\log z}{2}+\log z+\log(z+1)-\log(2z)\right) \\
= ~ & C_g\left(\log z-\log 2+\frac{z\eta_g}{z+1}+\log\left(1+\frac{1}{z}\right)\right)
+z\int_1^z\frac{E_g(t)}{(z+t)^2}dt.
\end{split}
\end{equation}
Now \eqref{eq21} follows from \eqref{eq23}.

If $\abs{E_g(t)}<c_g t^{-1/2}$ for $1\leq t\leq z$, then we have
\begin{equation}
\begin{split}
\abs{\int_1^z\frac{E_g(t)}{(z+t)^2}dt}
< ~ & c_g\int_1^z\frac{1}{t^{1/2}(z+t)^2}dt=c_g z\int_{\sqrt{1/z}}^1 \frac{2}{(1+u^2)^2}du \\
< ~ & \frac{c_g}{z^{3/2}}\left(\frac{1}{2}+\frac{\pi}{4}\right),
\end{split}
\end{equation}
which immediately gives \eqref{eq22}.
\end{proof}

Moreover, we use the following estimate for $\abs{E_k(t)}$ in order to improve our result for particular moduli.

\begin{lemma}\label{lm22}
Let $z_{k, i}$ and $c_{k, i}$ be real numbers given in Table \ref{tbl1}.
Define the multiplicative function $g(n)$ by $g(p)=0$ for a prime factor of $k$ and $g(p)=1/(p-1)$ for any
other primes.
Then, for each $i$, we have $\abs{E_g(t)}<c_{k, i}/t^{1/2}$ in the range $z_{k, i}<t\leq 10^8$.
\end{lemma}

\begin{proof}
This can be confirmed by calculation for each $k$.
\end{proof}

\begin{table}
\caption{$z_{k, i}$ and $c_{k, i}$ ($1\leq k\leq 16$)}
\begin{center}
\begin{small}
\begin{tabular}{| c || c | c | c | c | c | c | c | c |}
 \hline
$i$ & $z_{1, i}$ & $c_{1, i}$ & $z_{3, i}$ & $c_{3, i}$ & $z_{4, i}$ & $c_{4, i}$
& $z_{5, i}$ & $c_{5, i}$ \\
\hline
$1$ & $1$ & $1.45061$ & $1$ & $1.13253$ & $1$ & $0.83958$ & $1$ & $1.32358$ \\
$2$ & $3$ & $0.98839$ & $2$ & $0.84093$ & $2$ & $0.67357$ & $2$ & $1.24182$ \\
$3$ & $6$ & $0.91695$ & $3$ & $0.79347$ & $4$ & $0.40417$ & $3$ & $0.75765$ \\
$4$ & $7$ & $0.69135$ & $11$ & $0.46685$ & $12$ & $0.37998$ & $8$ & $0.60256$ \\
$5$ & $31$ & $0.54701$ & $15$ & $0.45082$ & $14$ & $0.36253$ & $34$ & $0.51047$ \\
$6$ & $44$ & $0.45440$ & $59$ & $0.38171$ & $16$ & $0.30942$ & $38$ & $0.47760$ \\
$7$ & $66$ & $0.41267$ & $71$ & $0.32698$ & $24$ & $0.27671$ & $39$ & $0.41380$ \\
$8$ & $67$ & $0.39083$ & $231$ & $0.31436$ & $30$ & $0.27290$ & $58$ & $0.38126$ \\
$9$ & $382$ & $0.38645$ & $240$ & $0.29411$ & $32$ & $0.26966$ & $394$ & $0.36632$ \\
$10$ & $391$ & $0.27702$ & $242$ & $0.28882$ & $34$ & $0.23733$ & $398$ & $0.35655$ \\
$11$ & $394$ & $0.26180$ & $371$ & $0.22641$ & $178$ & $0.20207$ & $399$ & $0.30457$ \\
$12$ & $475$ & $0.25440$ & $383$ & $0.20769$ & $196$ & $0.16347$ & $418$ & $0.27417$ \\
$13$ & & & & & & & $427$ & $0.25234$ \\
 \hline
 \hline
$i$ & $z_{7, i}$ & $c_{7, i}$ & $z_{8, i}$ & $c_{8, i}$ & $z_{9, i}$ & $c_{9, i}$ & $z_{11, i}$ & $c_{11, i}$ \\
\hline
$1$ & $1$ & $1.38049$ & $1$ & $0.83958$ & $1$ & $1.13253$ & $1$ & $1.47043$ \\
$2$ & $2$ & $1.37832$ & $2$ & $0.67357$ & $2$ & $0.84093$ & $3$ & $0.83347$ \\
$3$ & $3$ & $0.81744$ & $4$ & $0.40417$ & $3$ & $0.79347$ & $6$ & $0.70664$ \\
$4$ & $7$ & $0.60080$ & $12$ & $0.37998$ & $11$ & $0.46685$ & $7$ & $0.65915$ \\
$5$ & $30$ & $0.57462$ & $14$ & $0.36253$ & $15$ & $0.45082$ & $16$ & $0.64471$ \\
$6$ & $31$ & $0.51152$ & $16$ & $0.30942$ & $59$ & $0.38171$ & $44$ & $0.58033$ \\
$7$ & $58$ & $0.41797$ & $24$ & $0.27671$ & $71$ & $0.32698$ & $48$ & $0.53503$ \\
$8$ & $59$ & $0.39983$ & $30$ & $0.27290$ & $231$ & $0.31436$ & $174$ & $0.52900$ \\
$9$ & $66$ & $0.39849$ & $32$ & $0.26966$ & $240$ & $0.29411$ & $175$ & $0.50272$ \\
$10$ & $102$ & $0.38475$ & $34$ & $0.23733$ & $242$ & $0.28882$ & $178$ & $0.45745$ \\
$11$ & $103$ & $0.32344$ & $178$ & $0.20207$ & $371$ & $0.22641$ & $179$ & $0.43724$ \\
$12$ & $231$ & $0.31685$ & $196$ & $0.16347$ & $383$ & $0.20769$ & $182$ & $0.43107$ \\
$13$ & & & & & & & $183$ & $0.42345$ \\
$14$ & & & & & & & $390$ & $0.41919$ \\
$15$ & & & & & & & $391$ & $0.29608$ \\
$16$ & & & & & & & $394$ & $0.29557$ \\
 \hline
 \hline
$i$ & $z_{12, i}$ & $c_{12, i}$ & $z_{13, i}$ & $c_{13, i}$ &
$z_{15, i}$ & $c_{15, i}$ & $z_{16, i}$ & $c_{16, i}$ \\
\hline
$1$ & $1$ & $0.68179$ & $1$ & $1.48780$ & $1$ & $1.07770$ & $1$ & $0.83958$ \\
$2$ & $2$ & $0.48806$ & $3$ & $0.88960$ & $2$ & $0.78154$ & $2$ & $0.67357$ \\
$3$ & $6$ & $0.23883$ & $6$ & $0.77437$ & $3$ & $0.64540$ & $4$ & $0.40417$ \\
$4$ & $20$ & $0.23782$ & $7$ & $0.62091$ & $14$ & $0.50584$ & $12$ & $0.37998$ \\
$5$ & $54$ & $0.19122$ & $30$ & $0.60731$ & $15$ & $0.43390$ & $14$ & $0.36253$ \\
$6$ & $56$ & $0.18577$ & $31$ & $0.55600$ & $48$ & $0.36069$ & $16$ & $0.30942$ \\
$7$ & $60$ & $0.18309$ & $44$ & $0.47905$ & $49$ & $0.29725$ & $24$ & $0.27671$ \\
$8$ & $116$ & $0.17556$ & $48$ & $0.46724$ & $63$ & $0.27314$ & $30$ & $0.27290$ \\
$9$ & $120$ & $0.14573$ & $102$ & $0.46040$ & $134$ & $0.25295$ & $32$ & $0.26966$ \\
$10$ & $121$ & $0.14092$ & $103$ & $0.39232$ & & & $34$ & $0.23733$ \\
$11$ & & & $155$ & $0.36274$ & & & $178$ & $0.20207$ \\
$12$ & & & $202$ & $0.32107$ & & & $196$ & $0.16347$ \\
$13$ & & & $203$ & $0.31849$ & & & & \\
$14$ & & & $210$ & $0.30259$ & & & & \\
$15$ & & & $211$ & $0.28866$ & & & & \\
$16$ & & & $394$ & $0.26053$ & & & & \\
\hline
\end{tabular}
\label{tbl2}
\end{small}
\end{center}
\end{table}

\begin{lemma}\label{lm23}
Let $\Pi(z; k, a, Q)$ be the number of positive integers $n\leq z$ congruent to $a\Mod{k}$
and relatively prime to $Q$ for an integer $Q\geq 1$ and a real number $z>0$
and $\Pi(x, x+y; k, a, Q)=\Pi(x+y; k, a, Q)-\Pi(x; k, a, Q)$.
Moreover, we put $Q_r$ to be the product of the first $r$ primes and $Q_{k, r}=Q_r/\gcd(Q_r, k)$.
Then, for each integer $a$ and any positive real numbers $x$ and $y$, we have
\begin{equation}\label{eq24}
\Pi(x, x+y; k, a, Q_{k, r})<\frac{\vph(Q_{k, r})}{kQ_{k, r}}y+B_{k, r}
\end{equation}
where $B_{k, r}$'s are constants given in Tables \ref{tbl3a} and \ref{tbl3b}
and $B_{k, r}=B_{\ell, r}$ if $Q_{k, r}=Q_{\ell, r}$.
Furthermore, if $Q\geq 1$ is a squarefree integer whose largest primes factor is the $r$-th prime,
then
\begin{equation}
\Pi(x, x+y; k, a, Q)<\frac{\vph(Q)y}{Qk}+2a_i<\frac{\vph(Q_i)}{Q_i}\frac{y}{\vph(k)}+2B_{1, r}.
\end{equation}
\end{lemma}

\begin{proof}
For each modulus $k$ given in Tables \ref{tbl3a} and \ref{tbl3b}, we confirmed by calculation that
\begin{equation}
A_{k, r}+\frac{\vph(Q_{k, r})}{Q_{k, r}}\leq\Pi(N; Q_{k, r})-\frac{\vph(Q_{k, r})}{Q_{k, r}}N
\leq A_{k, r}+B_{k, r},
\end{equation}
where $A_{k, r}$'s are constants given in Tables \ref{tbl3a} and \ref{tbl3b} and
each equality holds for a certain integer $N$ and $\Pi(z; Q)=\Pi(z; 1, 0, Q)$.
Hence,
\begin{equation}
A_{k, r}<\Pi(z; Q_{k, r})-\frac{\vph(Q_{k, r})}{Q_{k, r}}z\leq A_{k, r}+B_{k, r},
\end{equation}
where both inequalities are strict.
We observe that $\Pi(z; k, 0, Q_{k, r})=\Pi(z/k; Q_{k, r})$ whenever $\gcd(Q_{k, r}, k)=1$ and therefore
\begin{equation}
A_{k, r}<\Pi(z; k, 0, Q_{k, r})-\frac{\vph(Q_{k, r})}{kQ_{k, r}}z\leq A_{k, r}+B_{k, r},
\end{equation}
which yields that
\begin{equation}
\Pi(x, x+y; k, 0, Q_{k, r})<\frac{\vph(Q_{k, r})}{kQ_{k, r}}y+B_{k, r}
\end{equation}
for any positive real numbers $x$ and $y$.

Now we observe that
\begin{equation}
\Pi(x, x+y; k, a, Q_{k, r})<\frac{\vph(Q_{k, r})}{kQ_{k, r}}y+B_{k, r}
\end{equation}
holds for each congruent class $a\Mod{k}$ and any positive real numbers $x$ and $y$.
Indeed, for each integer $b$, taking an integer $s$ such that $sQ_{k, r}\equiv -a\Mod{k}$, we have
\begin{equation}
\begin{split}
~ & \sup_{x\geq 0} \Pi(x, x+y; k, a, Q_{k, r}) \\
= ~ & \sup_{\alpha\geq -a/k} \# \{n: \alpha<n\leq \alpha+y/k, \gcd(kn+a, Q_{k, r})=1\} \\
= ~ & \sup_{\alpha\geq 0} \# \{n: \alpha<n\leq \alpha+y/k, \gcd(kn+(sQ_{k, r}+a), Q_{k, r})=1\} \\
= ~ & \sup_{\alpha\geq (sQ_{k, r}+a)/k} \# \{n: \alpha<n\leq \alpha+y/k, \gcd(kn, Q_{k, r})=1\} \\
= ~ & \sup_{\alpha\geq 0} \# \{n: \alpha<n\leq \alpha+y/k, \gcd(kn, Q_{k, r})=1\} \\
= ~ & \sup_{x\geq 0} \Pi(x, x+y; k, 0, Q_{k, r})=B_{k, r},
\end{split}
\end{equation}
which gives \eqref{eq24} for each congruent class $a\Mod{k}$.
Since $B_{k, r}$ depends only on $Q_{k, r}$, we have $B_{k, r}=B_{\ell, r}$ if $Q_{k, r}=Q_{\ell, r}$.
This proves the lemma.
\end{proof}

\section{Proof of the theorem}

Let $M$ and $N$ be positive integers.
For each prime $p$, let $\rho(p)$ distinct congruent classes $a_{p, 1}, \ldots, a_{p, \rho(p)}\Mod{p}$ be given.
Assume that a set $S$ of integers in the interval $[M+1, M+N]$ contains no integer
in distinct congruent classes $a_{p, 1}, \ldots, a_{p, \rho(p)}\Mod{p}$ for each prime $p$.

Using the result of \cite{Pre} based on Corollary 1 of \cite{MV}, we have
\begin{equation}\label{eq31}
\#S \leq\frac{N}{H_g(z)}
\end{equation}
with $z=\sqrt{(3/4)N}$ for $N\geq 2$,
where $g(n)$ is the squarefree-supported multiplicative function defined by $g(p)=\rho(p)/(p-\rho(p))$ for each prime $p$.

Let $M=\floor{(x-a)/k}$ and $N=\floor{(x+y-a)/k}-\floor{(x-a)/k}$.
If $p$ is a prime in the interval $\max\{x, z\}<p\leq x+y$ and $p\equiv a\Mod{k}$, then
$p=kn+a$ with $M<k\leq M+N$ and $kn+a\nequiv 0\Mod{p^\prime}$ for any prime $p^\prime\leq z$.
Thus, we can take $g(n)=g_k(n)$ to be the squarefree-supported multiplicative function defined like above
with $\rho(p)=1$ if $p\nmid k$ and $\rho(p)=0$ if $p\mid k$ and
\begin{equation}
H_{g_k}(z)=\sum_{q\leq z, \gcd(q, k)=1}\frac{\mu^2(q)}{\vph(q)}
\end{equation}
is equal to the left-hand side of \eqref{eq12}.

Henceforth we write $H_k(z)=H_{g_k}(z)$ and so on.
Now \eqref{eq31} yields that
\begin{equation}\label{eq32}
\pi(x, x+y; k, a)\leq \frac{N}{H_k(z)}+ \pi(z).
\end{equation}
Using Lemma \ref{lm21} with Alterman's estimate \eqref{eq12}-\eqref{eq12b}, we obtain
\begin{equation}
\begin{split}
H_k(z)> & ~ \frac{\vph(k)}{k}\left(\log z-\log 2+\frac{z}{z+1}\eta_k\right)
-\left(\frac{1}{2}+\frac{\pi}{4}\right)\frac{c_k}{z^{1/2}} \\
> & ~ \frac{\vph(k)}{2k}\left(\log N+2\eta_k-\log (16/3)-\epsilon_{k, 1}(z)\right).
\end{split}
\end{equation}
Hence, \eqref{eq31} and \eqref{eq32} yield that
\begin{equation}
\begin{split}
& \pi(x, x+y; k, a)\leq\frac{N}{H_k(z)}+z \\
& <\frac{2kN}{\vph(k)\left(\log N+2\eta_k-\log (16/3)-\epsilon_{k, 1}(z)-\epsilon_{k, 2}(N)\right)}.
\end{split}
\end{equation}
Observing that $\floor{y/k}\leq N<\floor{y/k}+1$, we have
\begin{equation}
\pi(x, x+y; k, a)\leq \frac{2y}{\vph(k)\left(\log(3ye^{2\eta_k}/(16k))-\epsilon_k(y)\right)}.
\end{equation}
This proves Theorem \ref{th11}.

We prove Theorem \ref{th12}.
Using Lemma 3 of \cite{MV}, we see that
\begin{equation}
H_k(z)\geq \frac{\vph(k)}{k}H_1(z).
\end{equation}
Hence, we obtain
\begin{equation}
\begin{split}
& \pi(x, x+y; k, a)\leq\frac{kN}{\vph(k)H_1(z)}+z \\
& <\frac{2kN}{\vph(k)\left(\log N+2\eta_1-\log (16/3)-\epsilon_{1, 1}(z)-\epsilon_{k, 2}^\prime(N)\right)},
\end{split}
\end{equation}
where
\begin{equation}
\epsilon_{k, 2}^\prime(N)=\frac{\vph(k)(\log N+2\eta_1-\log (16/3))^2}{k\sqrt{16N/3}},
\end{equation}
and therefore
\begin{equation}
\pi(x, x+y; k, a)\leq \frac{2y}{\vph(k)\left(\log(3ye^{2\eta_1}/(16k))-\epsilon_k^\prime(y)\right)},
\end{equation}
where
$\epsilon_k^\prime(y)=
\epsilon_{1, 1}(\sqrt{3\floor{y/k}/4})+\epsilon_{k, 2}^\prime(\floor{y/k})+\epsilon_{k, 3}^\prime(y)$ with
\begin{equation}
\epsilon_{k, 3}^\prime(y)=\frac{k(\log (y/k)+2\eta_1-\log (16/3))^2}{y}.
\end{equation}
From $\epsilon_{k, 2}^\prime(\floor{y/k})\leq \epsilon_{1, 2}(\floor{y/k})$ and
$\epsilon_{k, 3}^\prime(y)\leq \epsilon_{1, 3}(y/k)$,
we can easily observe that $\epsilon_k^\prime(y)\leq \epsilon_1(y/k)$.
Thus, we obtain
\begin{equation}
\pi(x, x+y; k, a)<\frac{2y}{\vph(k)(\log (y/k)+0.8601)}
\end{equation}
for $y\geq 10^{10} k$.

Using Lemma \ref{lm22}, we have for $z_{1, I}<z\leq 10^5$,
\begin{equation}
\begin{split}
\int_1^z\frac{E_1(t)}{(z+t)^2}dt
< ~ & \sum_{i=1}^{I-1} \int_{z_{1, i}}^{z_{1, i+1}}\frac{c_{1, i}}{t^{1/2}(z+t)^2}dt+\int_{z_{1, I}}^z \frac{c_{1, I}}{t^{1/2}(z+t)^2}dt \\
= ~ & \sum_{i=1}^{I-1} \int_1^{z_{1, i+1}}\frac{c_{1, i}-c_{1, i+1}}{t^{1/2}(z+t)^2}dt+\int_1^z \frac{c_{1, I}}{t^{1/2}(z+t)^2}dt \\
= ~ & \sum_{i=1}^{I-1} \frac{(c_{1, i}-c_{1, i+1})(\sqrt{z_{1, i+1}}-1)}{z^2}+\frac{c_{1, I}}{z^{3/2}}\left(\frac{1}{2}+\frac{\pi}{4}\right).
\end{split}
\end{equation}

Taking $I=12$, we see that if $475\leq z\leq 10^5$, then
\begin{equation}
\int_1^z\frac{E_1(t)}{(z+t)^2}dt\leq \frac{4.791}{z^2}+\frac{0.3271}{z^{3/2}}
\end{equation}
and
\begin{equation}
\abs{\Delta_1(z)}<\frac{0.3271}{z^{1/2}}+\frac{4.791}{z}.
\end{equation}

Hence, putting
\begin{equation}
\tilde \epsilon_{1, 1}(z)=\frac{\eta_1}{1+z}+\frac{0.3271}{z^{1/2}}+\frac{4.791}{z}
\end{equation}
and 
$\tilde \epsilon_k(y)=\tilde\epsilon_{1, 1}(\sqrt{3\floor{y/k}/4})
+\epsilon_{k, 2}^\prime(\floor{y/k})+\epsilon_{k, 3}^\prime(y)$,
we have
\begin{equation}
\pi(x, x+y; k, a)\leq \frac{2y}{\vph(k)\left(\log(3ye^{2\eta_k}/(16k))-\tilde \epsilon_k(y)\right)}
\end{equation}
for $669671k\leq y\leq 10^{10}k$.
This implies that
\begin{equation}\label{eq33}
\pi(x, x+y; k, a)<\frac{2y}{\vph(k)(\log (y/k)+0.8601)}
\end{equation}
for $669671k\leq y\leq 10^{10}k$.

Now it remains to prove \eqref{eq33} for $k<y\leq 669671k$.
It follows from Lemma \ref{lm23} that
\begin{equation}
\pi(x, x+y; k, a)\leq \Pi(x, x+y; k, a, Q_r)+r<\frac{\vph(Q_r)y}{Q_rk}+2B_{1, r}+r
\end{equation}
for $r=1, \ldots, 10$.
Observing that for any real $z>0$ and any integer $k$ not divisible by $31$,
\begin{equation}
\Pi(z; k, a, Q_{11})=\Pi(z; k, a, Q_{10})-\Pi(z/31;k, a, Q_{10})
\end{equation}
and so on, we have
\begin{equation}
\pi(x, x+y; k, a)<\frac{\vph(Q_{10+r})y}{Q_{10+r}k}+2^r B_{1, 10}+10+r
\end{equation}
for any positive integer $r$.

Hence, we obtain
\begin{equation}
\pi(x, x+y; k, a)<\frac{\vph(Q_{16})y}{Q_{16}k}+2^6 B_{1, 10}+16<\frac{2y}{\vph(k)(\log (y/k)+0.8601)}
\end{equation}
for $111557k\leq y\leq 669671k$ and
\begin{equation}
\pi(x, x+y; k, a)<\frac{\vph(Q_{10})y}{Q_{10}k}+2B_{1, 10}+10<\frac{2y}{\vph(k)(\log (y/k)+0.8601)}
\end{equation}
for $381k\leq y<111557k$.
Similarly, taking $r=3$, we have \eqref{eq33} for $19k\leq y<381k$ and,
taking $r=2$, we have \eqref{eq33} for $14k\leq y<19k$.

Now it remains to prove \eqref{eq33} for $k\leq y<14k$.
Assume that $13k\leq y<14k$.
If $\pi(x, x+y; k, a)\leq 7$, then, since $y/k\geq 13$, we have
\begin{equation}
\begin{split}
\pi(x, x+y; k, a)\leq ~ & 7<\frac{2y}{k(\log(y/k)+0.8601)} \\
\leq ~ & \frac{2y}{\vph(k)(\log(y/k)+0.8601)}.
\end{split}
\end{equation}
If $13k\leq y<14k$ and $\pi(x, x+y; k, a)\geq 8$, then 
$k$ must be divisible by $Q_5=2310$ and,
observing that $\vph(k)(0.8601+\log 14)<2k$,
\begin{equation}
\pi(x, x+y; k, a)\leq 13\leq \frac{y}{k}<\frac{2y}{\vph(k)(\log(y/k)+0.8601)}.
\end{equation}

Iterating this argument for other cases, we eventually obtain \eqref{eq33} for $k\leq y<14k$.
This proves Theorem \ref{th12}.

Similarly, for each $k$, Theorem \ref{th11} yields \eqref{eq14} for $y\geq 10^{10}k$.
Taking $I$, $d_1$, $d_2$, and $y_0$ from Table \ref{tbl4},
we have
\begin{equation}\label{eq34}
\abs{\Delta_k(z)}<\frac{d_1}{z^{1/2}}+\frac{d_2}{z}
\end{equation}
for $z_{k, I}<z\leq 10^5$ to obtain \eqref{eq14} for $y_0\leq y<10^{10}k$.
Moreover, Lemma \ref{lm23} with $k$ and $r$ taken from Table \ref{tbl5} yields that
\eqref{eq14} for $y_1\leq y\leq y_2$, where $y_1$ and $y_2$ are taken from the corresponding column of
Table \ref{tbl5}.
We can confirm \eqref{eq14} for remaining values of $y$ by calculation.
This proves Theorem \ref{th13}.

\begin{table}
\caption{$A_{k, r}$, $B_{k, r}$ ($1\leq k\leq 7$, $r\leq 10$)}
\begin{center}
\begin{small}
\begin{tabular}{| c | c | c | c |}
 \hline
$k$ & $r$ & $A_{k, r}$ & $B_{k, r}$ \\
\hline
$1$ & $1$ & $-1/2$ & $1$ \\
$1$ & $2$ & $-2/3$ & $4/3$ \\
$1$ & $3$ & $-14/15$ & $28/15$ \\
$1$ & $4$ & $-53/35$ & $106/35$ \\
$1$ & $5$ & $-194/77$ & $388/77$ \\
$1$ & $6$ & $-3551/1001$ & $7102/1001$ \\
$1$ & $7$ & $-92552/17017$ & $185104/17017$ \\
$1$ & $8$ & $-2799708/323323$ & $5599416/323323$ \\
$1$ & $9$ & $-9747144/676039$ & $19494288/676039$ \\
$1$ & $10$ & $-58571113/2800733$ & $117142226/2800733$ \\
\hline
$3$ & $3$ & $-4/5$ & $8/5$ \\
$3$ & $4$ & $-44/35$ & $88/35$ \\
$3$ & $5$ & $-151/77$ & $302/77$ \\
$3$ & $6$ & $-3044/1001$ & $6088/1001$ \\
$3$ & $7$ & $-75588/17017$ & $151176/17017$ \\
$3$ & $8$ & $-2467811/323323$ & $4935622/323323$ \\
$3$ & $9$ & $-8107159/676039$ & $16214318/676039$ \\
$3$ & $10$ & $-54863825/2800733$ & $109727650/2800733$ \\
\hline
$4$ & $2$ & $-2/3$ & $4/3$ \\
$4$ & $3$ & $-14/15$ & $28/15$ \\
$4$ & $4$ & $-53/35$ & $106/35$ \\
$4$ & $5$ & $-194/77$ & $388/77$ \\
$4$ & $6$ & $-3551/1001$ & $7102/1001$ \\
$4$ & $7$ & $-92552/17017$ & $185104/17017$ \\
$4$ & $8$ & $-2799708/323323$ & $5599416/323323$ \\
$4$ & $9$ & $-9747144/676039$ & $19494288/676039$ \\
$4$ & $10$ & $-58571113/2800733$ & $117142226/2800733$ \\
\hline
$5$ & $4$ & $-8/7$ & $16/7$ \\
$5$ & $5$ & $-164/77$ & $328/77$ \\
$5$ & $6$ & $-2815/1001$ & $5630/1001$ \\
$5$ & $7$ & $-74797/17017$ & $149594/17017$ \\
$5$ & $8$ & $-2249157/323323$ & $4498314/323323$ \\
$5$ & $9$ & $-7949104/676039$ & $15898208/676039$ \\
$5$ & $10$ & $-54009126/2800733$ & $108018252/2800733$ \\
\hline
$7$ & $5$ & $-53/33$ & $106/33$ \\
$7$ & $6$ & $-362/143$ & $724/143$ \\
$7$ & $7$ & $-10806/2431$ & $21612/2431$ \\
$7$ & $8$ & $-297163/46189$ & $594326/46189$ \\
$7$ & $9$ & $-1026369/96577$ & $2052738/96577$ \\
$7$ & $10$ & $-44136180/2800733$ & $88272360/2800733$ \\
\hline
\end{tabular}
\label{tbl3a}
\end{small}
\end{center}
\end{table}

\begin{table}
\caption{$A_{k, r}$, $B_{k, r}$ ($11\leq k\leq 15$, $r\leq 10$)}
\begin{center}
\begin{small}
\begin{tabular}{| c | c | c | c |}
 \hline
$k$ & $r$ & $A_{k, r}$ & $B_{k, r}$ \\
\hline
$11$ & $6$ & $-1126/455$ & $2252/455$ \\
$11$ & $7$ & $-29817/7735$ & $59634/7735$ \\
$11$ & $8$ & $-1013053/146965$ & $2026106/146965$ \\
$11$ & $9$ & $-31692926/3380195$ & $63385852/3380195$ \\
$11$ & $10$ & $-211063714/14003665$ & $422127428/14003665$ \\
\hline
$12$ & $3$ & $-4/5$ & $8/5$ \\
$12$ & $4$ & $-44/35$ & $88/35$ \\
$12$ & $5$ & $-151/77$ & $302/77$ \\
$12$ & $6$ & $-3044/1001$ & $6088/1001$ \\
$12$ & $7$ & $-75588/17017$ & $151176/17017$ \\
$12$ & $8$ & $-2467811/323323$ & $4935622/323323$ \\
$12$ & $9$ & $-8107159/676039$ & $16214318/676039$ \\
$12$ & $10$ & $-54863825/2800733$ & $109727650/2800733$ \\
\hline
$13$ & $7$ & $-5029/1309$ & $10058/1309$ \\
$13$ & $8$ & $-165135/24871$ & $330270/24871$ \\
$13$ & $9$ & $-552247/52003$ & $1104494/52003$ \\
$13$ & $10$ & $-3423099/215441$ & $6846198/215441$ \\
\hline
$15$ & $4$ & $-6/7$ & $12/7$ \\
$15$ & $5$ & $-109/77$ & $218/77$ \\
$15$ & $6$ & $-2386/1001$ & $4772/1001$ \\
$15$ & $7$ & $-58133/17017$ & $116266/17017$ \\
$15$ & $8$ & $-1786286/323323$ & $3572572/323323$ \\
$15$ & $9$ & $-5798301/676039$ & $11596602/676039$ \\
$15$ & $10$ & $-35448898/2800733$ & $70897796/2800733$ \\
\hline
\end{tabular}
\label{tbl3b}
\end{small}
\end{center}
\end{table}

\begin{table}
\caption{Constants related to \eqref{eq34}}
\begin{center}
\begin{small}
\begin{tabular}{| c | c | c | c | c |}
\hline
$k$ & $I$ & $d_1$ & $d_2$ & $y_0$ \\
\hline
$1$ & $12$ & $0.3271$ & $4.7910$ & $669671$ \\
$3$ & $12$ & $0.2670$ & $3.6639$ & $819829$ \\
$4$ & $12$ & $0.2102$ & $1.6947$ & $1261951$ \\
$5$ & $13$ & $0.3244$ & $4.0727$ & $1481037$ \\
$7$ & $12$ & $0.4073$ & $2.4615$ & $2467293$ \\
$8$ & $12$ & $0.2102$ & $1.6947$ & $2523903$ \\
$9$ & $12$ & $0.2670$ & $3.6639$ & $2549486$ \\
$11$ & $16$ & $0.3800$ & $4.9236$ & $4873274$ \\
$12$ & $10$ & $0.1812$ & $1.0737$ & $1559201$ \\
$13$ & $16$ & $0.3349$ & $3.9324$ & $5857271$ \\
$15$ & $9$ & $0.3252$ & $1.8005$ & $2213863$ \\
$16$ & $12$ & $0.2102$ & $1.6947$ & $4746576$ \\
\hline
\end{tabular}
\label{tbl4}
\end{small}
\end{center}
\end{table}

\begin{table}
\caption{Lemma \ref{lm23} yields \eqref{eq14} for $y_1\leq y\leq y_2$ ($k\leq 16$)}
\begin{center}
\begin{small}
\begin{tabular}{| c | c | c | c |}
\hline
$k$ & $r$ & $y_1$ & $y_2$ \\
\hline
$1$ & $16$ & $111557$ & $669671$ \\
$1$ & $10$ & $381$ & $111557$ \\
$1$ & $3$ & $19$ & $381$ \\
$1$ & $2$ & $14$ & $19$ \\
\hline
$3$ & $15$ & $82827$ & $819829$ \\
$3$ & $9$ & $463$ & $82827$ \\
$3$ & $3$ & $29$ & $463$ \\
$3$ & $1$ & $12$ & $29$ \\
\hline
$4$ & $16$ & $232135$ & $1261951$ \\
$4$ & $10$ & $753$ & $232135$ \\
$4$ & $3$ & $24$ & $753$ \\
$4$ & $1$ & $5$ & $24$ \\
\hline
$5$ & $15$ & $172030$ & $1481037$ \\
$5$ & $9$ & $939$ & $172030$ \\
$5$ & $2$ & $52$ & $939$ \\
\hline
$7$ & $15$ & $176230$ & $2467293$ \\
$7$ & $8$ & $753$ & $176230$ \\
$7$ & $2$ & $92$ & $753$ \\
\hline
$8$ & $16$ & $464270$ & $2523903$ \\
$8$ & $10$ & $1505$ & $464270$ \\
$8$ & $3$ & $48$ & $1505$ \\
$8$ & $1$ & $9$ & $48$ \\
\hline
\end{tabular}
\begin{tabular}{| c | c | c | c |}
\hline
$k$ & $r$ & $y_1$ & $y_2$ \\
\hline
$9$ & $15$ & $243117$ & $2549486$ \\
$9$ & $9$ & $1371$ & $243117$ \\
$9$ & $3$ & $82$ & $1371$ \\
$9$ & $1$ & $33$ & $82$ \\
\hline
$11$ & $16$ & $801217$ & $4873274$ \\
$11$ & $10$ & $2932$ & $801217$ \\
$11$ & $3$ & $212$ & $2932$ \\
$11$ & $2$ & $173$ & $212$ \\
\hline
$12$ & $15$ & $170643$ & $1559201$ \\
$12$ & $9$ & $898$ & $170643$ \\
$12$ & $3$ & $29$ & $898$ \\
$12$ & $1$ & $13$ & $29$ \\
\hline
$13$ & $16$ & $1048486$ & $5857271$ \\
$13$ & $10$ & $3697$ & $1048486$ \\
$13$ & $3$ & $257$ & $3697$ \\
$13$ & $2$ & $213$ & $257$ \\
\hline
$15$ & $15$ & $210062$ & $2213863$ \\
$15$ & $9$ & $1485$ & $210062$ \\
$15$ & $1$ & $51$ & $1485$ \\
\hline
$16$ & $15$ & $295860$ & $4746576$ \\
$16$ & $8$ & $1140$ & $295860$ \\
$16$ & $3$ & $94$ & $1140$ \\
$16$ & $1$ & $17$ & $94$ \\
\hline
\end{tabular}
\label{tbl5}
\end{small}
\end{center}
\end{table}

{}
\end{document}